\documentclass[final]{elsarticle}

\usepackage{amsmath,amsthm}
\usepackage{amssymb,latexsym}

\newtheorem{thm}{Theorem}[section]

\newtheorem{lma}[thm]{Lemma}
\newtheorem{cor}[thm]{Corollary}
\newtheorem{prp}[thm]{Proposition}

\newtheorem{clm}[thm]{Claim}

\def\Dmon{\Delta_{{\rm mon}}}
\begin{document}

\begin{frontmatter}
\title{Existences of rainbow matchings and rainbow matching covers.}

\author{Allan Lo\fnref{thanks}}
\address{School of Mathematics, University of Birmingham, Birmingham, B15~2TT, United Kingdom}
\ead{s.a.lo@bham.ac.uk}
\fntext[thanks]{The research leading to these results was supported by the European Research Council
under the ERC Grant Agreement no. 258345.}

\begin{abstract}
Let $G$ be an edge-coloured graph.
A rainbow subgraph in $G$ is a subgraph such that its edges have distinct colours.
The minimum colour degree $\delta^c(G)$ of $G$ is the smallest number of distinct colours on the edges incident with a vertex of~$G$.
We show that every edge-coloured graph $G$ on $n\geq 7k/2+2$ vertices with  $\delta^c(G) \geq k$ contains a rainbow matching of size at least $k$, which improves the previous result for $k \ge 10$.

Let $\Dmon(G)$ be the maximum number of edges of the same colour incident with a vertex of $G$. 
We also prove that if $t \ge 11$ and $\Dmon(G) \le t$, then $G$ can be edge-decomposed into at most $\lfloor tn/2 \rfloor $ rainbow matchings.
This result is sharp and improves a result of LeSaulnier and West.
\end{abstract}


\end{frontmatter}

\section{Introduction}

Let $G$ be a simple graph, that is, it has no loops or multi-edges. We write $V(G)$ for the vertex set of $G$ and $\delta(G)$ for the minimum degree of $G$. 
An \emph{edge-coloured graph} is a graph in which each edge is assigned a colour.
We say that an edge-coloured graph $G$ is \emph{proper} if no two adjacent edges have the same colour.
A subgraph $H$ of $G$ is \emph{rainbow} if all its edges have distinct colours.
Rainbow subgraphs are also called totally multicoloured, polychromatic, or heterochromatic subgraphs.

In this paper, we are interested in rainbow matchings in edge-coloured graphs.
The study of rainbow matchings began with a conjecture of Ryser~\cite{ryser}, which states that every Latin square of odd order contains a Latin transversal. 
Equivalently, for $n$ odd, every properly $n$-edge-colouring of $K_{n,n}$, the complete bipartite graph with $n$ vertices on each part, contains a rainbow copy of a perfect matching.
In a more general setting, given a graph $H$, we wish to know if an edge-coloured graph $G$ contains a rainbow copy of $H$.
A survey on rainbow matchings and other rainbow subgraphs in edge-coloured graphs can be found in \cite{kano}.

For a vertex $v$ of an edge-coloured graph $G$, the \emph{colour degree}, $d^c(v)$, of $v$ is the number of distinct colours on the edges incident with~$v$.
The smallest colour degree of all vertices in $G$ is the \emph{minimum colour degree of $G$} and is denoted by $\delta^c(G)$.
Note that a properly edge-coloured graph $G$ with $\delta(G) \geq k$ has $\delta^c(G) \geq k$. 

Li and Wang~\cite{wang2} showed that if $\delta^c(G)=k$, then $G$ contains a rainbow matching of size $\lceil (5k-3)/{12} \rceil$.
They further conjectured that if $k\geq 4$, then $G$ contains a rainbow matching of size $\lceil k/2 \rceil$.
% This bound is tight for properly edge-coloured complete graphs.
LeSaulnier et al.~\cite{lesaul} proved that if $\delta^c(G) = k$, then $G$ contains a rainbow matching of size $\lfloor k/2 \rfloor$.
% Furthermore, if $G$ is properly edge-coloured with $G\neq K_4$ or $|V(G)|\neq \delta(G) + 2$, then there is a rainbow matching of size $\big\lceil \frac{k}{2} \big\rceil$.
The conjecture was later proved in full by Kostochka and Yancey~\cite{kostochka2}.

Wang~\cite{wang1} asked does there exist a function $f(k)$ such that every properly edge-coloured graph $G$ on $n \ge f(k)$ vertices with $\delta(G) \ge k$ contains a rainbow matching of size at least $k$.
Diemunsch et al.~\cite{diemunsch2} showed that such function does exist and $f(k) \le 98k/23$.
Gy\'arf\'as and Sarkozy~\cite{gyarfas} improved the result to $f(k) \le 4k-3$.
Independently, Tan and the author~\cite{rainbow} showed that $f(k) \le 4k-4$ for $k \ge 4$.

Kostochka, Pfender and Yancey~\cite{kostochka1} showed that every (not necessarily properly) edge-coloured $G$ on $n \ge  17k^2/4$ vertices with  $\delta^c(G)\geq k$ contains a rainbow matching of size~$k$.
Tan and the author~\cite{rainbow} improved the bound to $n \geq 4k-4$ for $k \ge 4$.
In this paper we show that $n \ge 7k/2 +2$ is sufficient.

\begin{thm} \label{rainbowmatching}
Every edge-coloured graph $G$ on $n \ge 7k/2+2$ vertices with $\delta^c(G) \geq k$ contains a rainbow matching of size $k$. 
\end{thm}

Moreover if $G$ is bipartite, then we further improve the bound to $n \ge (3 + \varepsilon) k + \varepsilon^{-2}$.

\begin{thm} \label{rainbowmatching2}
Let $0< \varepsilon \le 1/2$ and $k \in \mathbb{N}$.
Every edge-coloured bipartite graph $G$ on $n \ge (3+ \varepsilon)k + \varepsilon^{-2}$ vertices with $\delta^c(G) \geq k$ contains a rainbow matching of size~$k$. 
\end{thm}

We also consider covering an edge-coloured graph $G$ by rainbow matchings.
Given an edge-coloured graph $G$, let $\Dmon(G)$ be the largest maximum degree of monochromatic subgraphs of~$G$.
LeSaulnier and West~\cite{LeSaulnierWest} showed that every edge-coloured graph $G$ on $n$ vertices with $\Dmon(G) \le t$ has an edge-decomposition into at most $t(1+t) n \ln n$ rainbow matchings.
We show that $G$ can be edge-decomposed into $\lfloor t n/2 \rfloor$ rainbow matchings provided $t \ge 11$. 

\begin{thm} \label{thm:rainmatchigndecom}
For all $t \ge 11$, every edge-coloured graph $G$ on $n$ vertices with $\Dmon(G) \le t$ can be edge-decomposed into $\lfloor t n/2 \rfloor$ rainbow matchings.
\end{thm}
Note that the bound is best possible by considering edge-coloured graphs, where one colour class induces a $t$-regular graph.

Theorems~\ref{rainbowmatching} and~\ref{rainbowmatching2} are proved in Section~\ref{sec:rainbow}.
Theorem~\ref{thm:rainmatchigndecom} is proved in Section~\ref{sec:decom}.

\section{Existence of rainbow matchings} \label{sec:rainbow}

We write $[k]$ for $\{1,2,\ldots,k\}$.
Let $G$ be a graph with an edge-colouring~$c$. 
We denote by $c(G)$ the set of colours in $G$.
We write $|G|$ for $|V(G)|$.
Given $W \subseteq V(G)$, $G[W]$ is the induced subgraph of $G$ on~$W$.
All colour sets are assumed to be finite.

Before proving Theorems~\ref{rainbowmatching} and~\ref{rainbowmatching2}, we consider the following (weaker) question.
Suppose that $G$ is an edge-coloured graph and contains a rainbow matching $M$ of size $k-1$.
Under what colour degree and $|G|$ conditions can we `extend' $M$ into a matching of size $k$ with at least $k-1$ colours?
We formalise the question below. 

Let $\mathcal{G}$ be a family of graphs closed under vertex/edge deletions.
Define $\gamma(\mathcal{G})$ to be the smallest constant $\gamma$ such that, whenever $k \in \mathbb{N}$, $G \in \mathcal{G}$ is a graph with $|G| \ge \gamma k$ and an edge-colouring~$c$ on $G$, the following holds.
If for any rainbow matching $M$ of size~$k-1$ in~$G$, we have $d^c(z) \ge k$ for all $z \in V(G) \setminus V(M)$, then $G$ contains a rainbow matching $M'$ of size $k-1$ and a disjoint edge.
(Note that the colour of the disjoint edge may appear in $M'$.)
Clearly, $\gamma(\mathcal{G}) \ge 2$ for any family $\mathcal{G}$ of graphs.
It is easy to see that equality holds if $\mathcal{G}$ is the family of bipartite graphs.

\begin{prp} \label{prp:alpha}
Let $\mathcal{G}$ be the family of bipartite graphs.
Then $\gamma(\mathcal{G}) =2$.
\end{prp}

\begin{proof}
Let $G$ be a bipartite graph on at least $2k$ vertices. 
Suppose that $M$ is a rainbow matching of size~$k-1$ and that $d^c(z) \ge k$ for all $z \in V(G) \setminus V(M)$.
Since $G$ is bipartite, there exists an edge vertex-disjoint from $M$ and so the proposition follows. 
\end{proof}

If $\mathcal{G}$ is the family of all graphs, we will show that $\gamma(\mathcal{G}) \le 3$.

\begin{lma}
Let $G$ be a graph with at least $3(k-1)+1$ vertices.
Suppose that $M$ is a rainbow matching of size~$k-1$ and that $d^c(z) \ge k$ for all $z \in V(G) \setminus V(M)$.
Then $G$ contains a rainbow matching $M'$ of size $k-1$ and a disjoint edge.
\end{lma}

\begin{proof}

Let $x_1y_1, \ldots, x_{k-1} y_{k-1}$ be the edges of~$M$ with $c(x_iy_i) = i$.
Let $W = V(G) \setminus V(M)$.
We may assume that $G[W]$ is empty or else the lemma holds easily.

Suppose the lemma does not hold for~$G$.
By relabeling the indices of $i$ and swapping the roles of $x_i$ and $y_i$ if necessary, we will show that there exist distinct vertices $z_{1}$, \dots, $z_{k-1}$ in $W$ such that for each $1 \le i \le k-1$, the following holds:
\begin{enumerate}
 \item[(a$_i$)] $y_i z_i$ is an edge and $c(y_iz_i) \notin [i]$.
 \item[(b$_i$)] Let $T_i$ be the vertex set $\{x_j, y_j, z_j : i \le j \le k-1\}$.
       For any colour $j'$, there exists a rainbow matching $M^i_{j'}$ of size $k-i$ on $T_i$ such that $ c(M^i_{j'}) \cap ([i-1] \cup \{j'\}) = \emptyset$.
 \item[(c$_i$)] Let $W_i = W \backslash \{z_i, z_{i+1}, \dots, z_{k-1} \}$.
       For all $w \in W_i$, $N(w) \cap T_i \subseteq \{ y_i, \dots, y_{k-1}\}$.

\end{enumerate} 
% \textit{Proof of Claim.} 
Let $W_k=W$ and $T_k=\emptyset$.
Suppose that we have already found $z_{k-1}, z_{k-2}, \ldots, z_{i+1}$.
We find $z_i$ as follows.

Note that $|W_{i+1}| \ge n - 2(k-1) - (k-i-1) \ge 1$, so $W_{i+1} \ne \emptyset$.
Let $z$ be a vertex in $W_{i+1}$. 
By the colour degree condition, $z$ must incident to at least $k$ edges of distinct colours, and in particular, at least $k-i$ distinct coloured edges not using colours in $[i]$. 
%If $i=k-1$, then there exists a vertex $u \in V(M\cup M')$ such that $uz$ is an edge with $c(uz)\notin [k-1]$, and clearly, $u\notin S$, as otherwise $G$ would contain a rainbow matching of size $k$. 
By~(c$_{i+1}$), $z$ sends at most $k-i-1$ edges to~$T_{i+1}$. 
So there exists a vertex $u \in V(M) \setminus T_{i+1} = \{x_j,y_j: 1\leq j \leq i\}$ such that $uz$ is an edge with $c(uz)\notin [i]$.
Without loss of generality, $u=y_i$ and we set $ z_i=z $. 
Clearly (a$_i$) holds.

We now show that (b$_i$) holds for any colour $j'$.
If $j' \neq i$, then by (b$_{i+1}$), there is a rainbow matching $M^{i+1}_{j'}$ of size $k-i-1$ on $T_{i+1}$ such that  $c(M^{i+1}_{j'}) \cap ([i] \cup \{j'\}) = \emptyset$.
Set $M^{i}_{j'} = M^{i+1}_{j'} \cup x_iy_i$.
So $M^{i}_{j'}$ is a rainbow matching on $T_i$ of size $k-i$ and moreover $c(M^{i}_{j'}) \cap ([i-1] \cup \{j'\}) = \emptyset$ as required.
If $j' = i$, then by~(b$_{i+1}$), there is a rainbow matching $M^{i+1}_{c(y_iz_i)}$ of size $k-i-1$ on $T_{i+1}$ such that  $c(M^{i+1}_{c(y_iz_i)}) \cap ([i] \cup \{c(y_iz_i)\}) = \emptyset$.
Set $M^{i}_{i} = M^{i+1}_{c(y_iz_i)} \cup y_iz_i$.
Note that $M^{i}_i$ is the desired rainbow matching.

%Part (c)
Let $w t$ be an edge with $w \in W_i$ and $t \in T_i$. 
Since $G[W]$ is empty, $t \notin \{z_i, z_{i+1}, \dots, z_{k-1}\}$.
By~(c$_{i+1}$), $t \notin \{x_{i+1}, x_{i+2},\dots, x_{k-1}\}$.
Suppose that $t = x_i$. 
By~(b$_{i+1}$), there exists a rainbow matching $M^{i+1}_{c(y_iz_i)}$ of size $k-i-1$ on $T_{i+1}$ such that $c(M^{i+1}_{c(y_iz_i)}) \cap ( [i]\cup \{c(y_iz_i)\}) = \emptyset$.
Let $M'$ be the matching $\{x_j y_j : j \in [i-1] \} \cup M^{i+1}_{c(y_iz_i)} \cup \{y_i z_i\}$.
Note that $M'$ is a rainbow matching of size $k-1$ vertex-disjoint from the edge~$w x_i$. 
This contradicts the fact that $G$ is a counterexample.
Hence we have $t \in \{y_i, y_{i+1}, \dots, y_{k-1}\}$ implying~(c$_i$).

Therefore we have found $z_1, \dots, z_{k-1}$. 
Let $w \in W_1 \ne \emptyset$.
Recall the $G[W] = \emptyset$, so $N(w) \subseteq \{y_1, \dots, y_{k-1}\}$ by~(c$_1$), which implies that $d^c(w) \le d(w) \le k-1$, a contradiction. 
\end{proof}

\begin{cor} \label{cor:alpha}
Every family $\mathcal{G}$ of graphs satisfies $\gamma(\mathcal{G}) \le 3$.
\end{cor}

For colour sets~$C$ and integers~$\ell$, we now define a $(C,\ell)$-adapter below, which will be crucial in the proof of Lemma~\ref{lma:key}.
Roughly speaking a $(C, \ell)$-adapter is a vertex subset $W$ that contains a rainbow matching $M$ with $c(M) = C$ even after removing a vertex in~$W$.

Given $\ell \in \mathbb{N}$ and a set $C$ of colours, a vertex subset $W \subseteq V(G)$ is said to be a \emph{$(C,\ell)$-adapter} if there exist (not necessarily edge-disjoint) rainbow matchings $M_1, \ldots, M_{\ell}$ in $G[W]$ such that $c(M_i) = C$ for all $i \in [\ell]$, and  given any $w \in W$, there exists $i \in [\ell]$ such that $w \notin V(M_i)$. 
We write $C$-adapter for $( C, |C|+1)$-adapter. 
Note that a $(C,\ell)$-adapter is also a $(C,\ell')$-adapter for all $\ell \le \ell'$.
The following proposition studies some basic properties of $(C,\ell)$-adapters.

\begin{prp} \label{prp:plug}
Let $G$ be a graph with an edge-colouring~$c$. 
\begin{enumerate}
	\item[\rm (i)] Let $C= \{c_1, \dots, c_{\ell} \}$ be a set of distinct colours.
	Let $W = \{x_i,y_i, z_i, w : i \in [\ell] \}$ be a vertex set such that $c(x_iy_i) = c_i=  c(z_iw)$ for all $i \in [\ell]$.
	Then $W$ is a $C$-adapter.
	
	\item[\rm (ii)] Let $\ell_1, \ldots, \ell_p \in \mathbb{N}$ and let $C_1, \ldots, C_{p}$ be pairwise disjoint colour sets.
	Suppose that $W_j$ is a $(C_j, \ell_j)$-adapter for all $j \in [p]$ and that $W_1, \ldots, W_p$ are pairwise disjoint.
	Then $\bigcup_{j=1}^{p} W_j$ is a $(\bigcup_{j=1}^{p} C_j, \max_{ j \in [p] } \{\ell_j\} )$-adapter.

	\item[\rm (iii)] Let $C$ be a colour set.
	Suppose that $W$ is a $(C, \ell)$-adapter.
	Suppose that $x,y,z \in V(G) \setminus W$ and $w \in W$ such that $xy, zw \in E(G)$ and $c(xy) = c(zw) \notin C$.
	Then $W \cup \{x,y,z\}$ is a $( C \cup \{c(xy) \}, \ell+1 )$-adapter.
\end{enumerate}
\end{prp}

\begin{proof}
To prove (i), we simply set $M_i = \{x_jy_j :  j \in [\ell] \setminus \{i\} \} \cup \{wz_i\}$ for all $i \in [\ell]$ and $M_{\ell +1} = \{x_jy_j :  j \in [\ell]\}$.

(ii) Let $\ell = \max \{\ell_j : j \in [p] \}$.
Note that each $W_j$ is a $(C_j, \ell)$-adapter.
For $j \in [p]$, let $M_1^j, \ldots, M_{\ell}^j$ be rainbow matchings in $G[W_j]$ such that $c(M_i^j) = C_j$ for all $i \in [\ell]$, and given any $w \in W_j$, there exists $i \in [\ell]$ such that $w \notin V(M_i^j)$. 
Set $M_i = \bigcup_{j=1}^{p} M^j_i$.
So (ii) holds.

(iii)
Let $M_1,\dots, M_{\ell}$ be rainbow matchings in~$G[W]$ such that $c(M_i) = C$ for all $i \in [\ell]$, and given any $w' \in W$, there exists $i \in [\ell]$ such that $w' \notin V(M_i)$. 
Without loss of generality we have $w \notin V(M_1)$.
Now set $M'_i = M_i \cup \{xy\} $ for all $i \in [\ell]$ and $M'_{\ell+1} = M'_1 \cup \{wz\}$.
Hence, $W \cup \{x,y,z\}$ is a $( C \cup \{c(xy) \}, \ell+1 )$-adapter.
\end{proof}

We prove the following lemma.
The main idea of the proof is to consider $(C,\ell)$-adapters in~$G$ with $\ell$ maximal.

\begin{lma} \label{lma:key}
Let $k \in \mathbb{N}$ and let $2< \gamma \le 3$.
Let $\mathcal{G}$ be a family of graphs closed under vertex/edge deletion with $\gamma (\mathcal{G}) \le \gamma$.
Suppose that $G \in \mathcal{G}$ with 
\begin{align*}
|G| \ge \left( 2+ \frac{\gamma}2 \right)k  +\frac{2(4-\gamma)}{(\gamma - 2)^2}-3 + \gamma
\end{align*}
and that $G$ contains a rainbow matching of size $k-1$.
Further suppose that for all rainbow matchings $M$ of size $k-1$ in $G$, we have $d^c(v) \ge k$ for all $v \in V(G) \setminus V(M)$.
Then $G$ contains a rainbow matching of size~$k$.
\end{lma}

\begin{proof}
We proceed by induction on~$k$.
It is trivial for $k=1$, so we may assume that $k \ge 2$.

Let $p \in \mathbb{N} \cup \{0\}$ and let $\ell_1, \ldots, \ell_p \in \mathbb{N}$ with $\ell_1 \ge \ldots \ge \ell_p$ and $\sum_{i=1}^{p} \ell_i \le k-1$.
Let $\mathcal{P} = \{W_1, \ldots, W_p ,U\}$ be a vertex partition of $V(G)$.
We say that $\mathcal{P}$ has parameters $(\ell_1, \ell_2, \ldots, \ell_p)$ if 
\begin{enumerate}
	\item[\rm (a)] there exist $p$ pairwise disjoint colour sets $C_1, \ldots, C_p$ such that  $|C_i| = \ell_i$ for all $ i \in [p]$;
	\item[\rm (b)] $W_i$ is a $C_{i}$-adapter and $|W_i| = 3 \ell_i + 1$ for all $i \in [p]$;
	\item[\rm (c)] there exists a rainbow matching $M_U$ of size $k-1 - \sum_{i=1}^{p} \ell_i$ in $G[U]$ with $c(M_U) \cap C_i = \emptyset$ for all $i \in [p]$;
 	\item[\rm (d)] $U  \setminus V(M_U) \ne \emptyset$.
\end{enumerate}
Since $G$ contains a rainbow matching $M$ of size $k-1$, such a vertex partition exists ($p=0$ and $U = V(G)$ say).
We now assume that $\mathcal{P}$ is chosen such that the string $(\ell_1, \ldots, \ell_p)$ is lexicographically maximal. 
(Here, we view $(a_1,a_2,\dots,a_p)$ as $(a_1,a_2,\dots,a_p,0, \dots, 0)$, e.g. $(3,2,2) \le (4,1) \le (4,1,1)$.)

Let $C_1, \dots, C_p$ be the sets of colours guaranteed by (a)--(c).
Set $W = W_1 \cup \ldots \cup W_p$ and  $C= \bigcup_{i = 1}^{p} C_i$.
Let $\ell_0 = k-1 - \sum_{i=1}^p \ell_i$.
By (b) and Proposition~\ref{prp:plug}(ii), $W$ is a $(C, \ell_1+1)$-adapter.
The following claim gives some useful properties of the rainbow matchings in $G[U]$ and~$G \setminus W$.
This will be needed to finish the proof of the lemma.

\begin{clm} \label{clm:switch}
\begin{itemize}
\item[\rm (i)] Let $M_U$ be a rainbow matching of size $\ell_0$ in $G[U]$ with $c(M_U) \cap C = \emptyset$.
If $|U| \ge 2 \ell_0 + 2$ and there is an edge $wz \in E(G)$ with $w \in W$ and $z \in U \setminus V(M_U)$, then we have $c(wz) \in C$. 

\item[\rm (ii)] Let $M'$ be a rainbow matching of size $k-1- \ell_1$ in $G \setminus W$ with $c(M') \cap C_1 = \emptyset$.
If $wx \in E(G)$ with $w \in W_1$ and $x \in V(G) \setminus (W_1 \cup V(M'))$, then $c(wx) \in C_1$. 
\end{itemize}
\end{clm}

\begin{proof}[Proof of Claim]
Suppose that (i) is false. 
There exists an edge $wz \in E(G)$ such that $c(wz) \notin C$, $w \in W_i$ for some $i \in [p]$ and $z \in U \setminus V(M_U)$.
Note that there exists a rainbow matching $M_W$ in $G[W \setminus w]$ such that $c(M_W) = C$ since $W$ is a $C$-adapter.
If $c(wz) \notin C \cup c(M_U)$, then $M_U \cup M_W \cup \{wz\}$ is a rainbow matching of size~$k$, so we are done.
If $c(wz) \in c(M_U)$, then let $xy$ be the edge in $M_U$ such that $c(xy) = c(wz)$.
Set $W'_i = W_i \cup\{x,y,z \}$, $W'_j = W_j$ for all $j \in [p] \setminus \{i\}$ and $U' = U \setminus \{x,y,z\}$.
Let $\ell_i' = \ell_i +1$ and let $\ell_j' = \ell_j$ for all $j \in [p] \setminus \{i\}$.
Set $C'_i = C_i \cup\{ c(xy) \}$ and $C'_j = C_j$ for all $j \in [p] \setminus \{i\}$.
By Proposition~\ref{prp:plug}(iii), $W'_j$ is a $C'_j$-adapter for all $j \in [p]$.
Note that $M_{U'} = M_U - xy$ is a rainbow matching in $G[U']$ with $c(M_{U'}) \cap C'_j = \emptyset$ for all $j \in [p]$.
Also $U' \setminus V( M_{U'} )  = U \setminus ( V(M_U) \cup \{z\}) \ne \emptyset$.
By relabelling the sets $W'_j$ and $C'_j$ if necessary, we deduce that the vertex partition $\mathcal{P}' = \{W'_1, \dots, W'_p, U'\}$ has parameters $(\ell'_1, \dots, \ell'_p) > (\ell_1, \dots, \ell_p)$, which contradicts the maximality of $\mathcal{P}$.
Hence (i) holds.

A similar argument proves~(ii).
\end{proof}

Suppose that $|U| > \gamma (\ell_0+1)$, so $|U| \ge 2 \ell_0 +3$.
Let $H$ be the resulting subgraph of $G [U]$ obtained after removing all edges of colours in~$C$. 
Let $M_U$ be a rainbow matching in $H$ of size~$\ell_0$ with $c(M_U) \cap C = \emptyset$, which exists by~(c).
By Claim~\ref{clm:switch}(i), we have for all $z \in V(H)\setminus V(M_U)$, $d_H^c(z) \ge k - |C| = \ell_0 +1$.
Since $\gamma(\mathcal{G}) \le \gamma$, $H$ contains a rainbow matching $M_0$ of size $\ell_0$ and a disjoint edge~$e$.
If $c(e) = c(xy)$ for some $xy \in M_0$, then set $W_{p+1} = V(e) \cup \{x,y\}$, $C_{p+1} = \{c(xy)\}$, and $U' = U \setminus ( V(e) \cup \{x,y\} ) $.
Observe that $W_{p+1}$ is a $C_{p+1}$-adapter by Proposition~\ref{prp:plug}(i).
Note that $M_0 - xy$ is a rainbow matching of size $\ell_0 -1$ in $G[U']$ with $c(M_0) \cap \bigcup_{j \in [p+1]}C_j = \emptyset$ and $|U' \setminus V(M_0) | = |U| - 2\ell_0 - 2 \ge 1$. 
Hence the vertex partition $\mathcal{P}'=\{ W_1, \dots, W_{p+1}, U'\}$ has parameters $(\ell_1, \dots, \ell_p,1)$, contradicting the maximality of~$\mathcal{P}$.
If $c(e) \notin c(M_0)$, then $M_0 \cup e$ is a rainbow matching with $c(M_0 \cup e) \cap C = \emptyset$.
Together with~(b), $G$ contains a rainbow matching of size~$k$ with colours $c(M_0 \cup e) \cup C$, so we are done.
Therefore we may assume that 
\begin{align}
|U| \le \gamma (\ell_0+1). \label{eqn:Z}
\end{align}
Since $2 <\gamma \le 3$ and $\ell_0 \le k-1$, by the assumptions of Lemma~\ref{lma:key}, we have $|G| > (2+\gamma/2)k > \gamma k \ge |U|$.
Therefore, $W \ne \emptyset$ and $\ell_1 \ge 1$.

Next, suppose that $( \gamma -2) \ell_1 \ge 2$, so $|W_1| = 3 \ell_1 +1 \le (2+ \gamma/2) \ell_1$.
Let $H_1$ be the subgraph of $G$ obtained by removing all vertices of $W_1$ and all edges of colours in $C_1$.
By the assumptions of Lemma~\ref{lma:key}, we then have
\begin{align*}
|H_1| = |G| - |W_1|  \ge \left( 2+ \frac{\gamma}2 \right) (k - \ell_1)  +\frac{2(4-\gamma)}{(\gamma - 2)^2}-3 + \gamma.
\end{align*}
By (b) and~(c), $H_1$ contains a rainbow matching~$M'$ of size $k-1- \ell_1$.
By Claim~\ref{clm:switch}(ii), $c(wx) \in C_1$ for all $w \in W_1$ and $x \in V(H_1) \setminus V(M')$.
Hence, $d_{H_1}^c(z) \ge k - |C_1| = k - \ell_1$ for all $z \in V(H_1) \setminus V(M')$.
Note that this statement also holds for any rainbow matchings $M'$ of size $k-1-\ell_1$ in $H_1$.
Hence $H_1$ satisfies the hypothesis of the lemma with $k = k- \ell_1$. 
By the induction hypothesis, $H_1$ contains a rainbow matching~$M''$ of size~$k - \ell_1$.
By~(b), there exists a rainbow matching~$M_1$ of size~$\ell_1$ in $G[W_1]$ such that $c(M_1) = C_1$.
Since $c(M_1) \cap c(M'') \subseteq C_1 \cap c(H_1) = \emptyset$, $M_1 \cup M''$ is a rainbow matching of size~$k$ as required.
Therefore we may assume that 
\begin{align}
( \gamma -2) \ell_1 < 2. \label{eqn:alpha}
\end{align}

Recall that $W$ is a $(C, \ell_1+1)$-adapter. 
So there exist rainbow matchings $M_1^*$, $M_2^*$, \dots, $M_{\ell_1+1}^*$ such that $c(M_i^*) = C$ for all $i \in [\ell_1+1]$ and 
\begin{align}
W = \bigcup_{i = 1}^{\ell_1 +1} (W \setminus V(M_i^*)). \label{eqn:W}
\end{align}
Let $M_U$ be a rainbow matching of size~$\ell_0$ in $G[U]$ with $c(M_U) \cap C = \emptyset$ (which exists by (c)).
By~(d), there exists $z \in U \setminus V(M_U)$. 
Note that $z$ sends at least $d^c(z) - |V(M_U)| \ge k - 2\ell_0 $ edges of distinct colours to $V(G) \setminus V(M_U)$.
Let $q = \lceil (k- 2 \ell_0) / (\ell_1 +1) \rceil$.
By~\eqref{eqn:W} and an averaging argument, there exists $i \in [\ell_1 +1]$ such that there exist vertices $x_1, \dots, x_q \in V(G) \setminus V( M_U \cup M_i^* )$ such that $c(z x_j)$ is distinct for each $j \in [q]$.
By Claim~\ref{clm:switch}(i), we have  $c( z x_j) \in C = c(M_i^*)$ for all $j \in [q]$.
Let $e_1, \ldots, e_q$ be edges of $M_i^*$ such that $c(e_j) = c(zx_j)$ for all $j \in [q]$.
Set $W' = \bigcup_{j \in [q]} ( V(e_j) \cup \{x_j,z\})$ and $ C' = \{ c(e_j) : j \in[q]\}$.
By Proposition~\ref{prp:plug}(i), $W'$ is a $C'$-adapter.
Set $U' = V(G) \setminus W'$ and $M_{U'} = ( M_i^* \cup M_U ) \setminus W'$.
Note that $V( M_{U'} ) \subseteq U'$ and $M_{U'}$ is a rainbow matching of size $k-1 - q$ with $c(M_{U'}) \cap C' = \emptyset$.
Therefore, the vertex partition $\mathcal{P}' = \{W',U'\}$ has parameter~$(q)$.
By the maximality of $\mathcal{P}$, we have $\ell_1 \ge q \ge (k- 2 \ell_0) / (\ell_1 +1) $ and so
\begin{align}
\ell_0 \ge  (k - \ell_1(\ell_1 +1))/2. \label{eqn:1}
\end{align}
Recall that $|W_i| = 3 \ell_i +1 \le 4 \ell_i$ for all $i \in [p]$, that $\sum_{i=1}^p \ell_i + \ell_0 = k-1$, and that $2 < \gamma \le 3$.
Finally, we have 
\begin{align*}
	|G| &= |W_1| + \sum_{i = 2}^{p}|W_i| + |U|
	 \overset{\eqref{eqn:Z}}{\le} 3\ell_1 +1 + 4 \sum_{i = 2}^{p} \ell_i + \gamma ( \ell_0 +1) \\
	& = 3\ell_1 +1 + 4(k-1- \ell_1)  -(4-\gamma) \ell_0 + \gamma\\
	& \overset{\eqref{eqn:1}}{\le}  4k-3- \ell_1  - \frac{(4-\gamma) (k - \ell_1(\ell_1 +1))}{2} + \gamma \\
 	& = \left( 2+ \frac{\gamma}2 \right)k -3 - \ell_1 +\frac{(4-\gamma)\ell_1(\ell_1 +1)}{2} + \gamma\\
 	& < \left( 2+ \frac{\gamma}2 \right)k  +\frac{(4-\gamma)\ell_1^2}{2}-3 + \gamma 
	 \overset{\eqref{eqn:alpha}}{<}  \left( 2+ \frac{\gamma}2 \right)k  +\frac{2(4-\gamma)}{(\gamma - 2)^2}-3 + \gamma,
\end{align*}
a contradiction. 
This completes the proof of the lemma.
\end{proof}

We are now ready to prove Theorems~\ref{rainbowmatching} and~\ref{rainbowmatching2}.

\begin{proof}[Proof of Theorems~\ref{rainbowmatching} and~\ref{rainbowmatching2}]
We first prove Theorem~\ref{rainbowmatching} by induction on~$k$.
Let $G$ be an edge-coloured graph on $n \ge 7k/2+2$ vertices with $\delta^c(G) \ge k$.
This is trivial for $k =1$ and so we may assume that $k \ge 2$.
By the induction hypothesis $G$ contains a rainbow matching of size $k-1$.
Since $\delta^c(G) \ge k$, Corollary~\ref{cor:alpha} implies that $G$ satisfies the hypothesis of Lemma~\ref{lma:key} with $\gamma = 3$.
Therefore, $G$ contains a rainbow matching of size $k$ as required.

To prove Theorem~\ref{rainbowmatching2}, first note that by Proposition~\ref{prp:alpha}, $\gamma(\mathcal{G}')= 2$, where $\mathcal{G}'$ is the family of all bipartite graphs. 
Also, for $\gamma = 2 + 2 \varepsilon$, we have 
\begin{align*}
\left( 2+ \frac{\gamma}2 \right)k  +\frac{2(4-\gamma)}{(\gamma - 2)^2}-3 + \gamma
 & = (3+ \varepsilon)k + \frac{2(2- 2\varepsilon)}{4\varepsilon^{2}} - 1 + 2\varepsilon \le (3+ \varepsilon)k + \varepsilon^{-2}.
\end{align*}
Therefore, Theorem~\ref{rainbowmatching2} follows from a similar argument used in the preceding paragraph, where we take $\gamma = 2 + 2 \varepsilon$ and $\mathcal{G}$ to be the family of all bipartite graphs in the application of Lemma~\ref{lma:key}.
\end{proof}

We would like to point out that an improvement of Corollary~\ref{cor:alpha} would lead to an improvement of Theorem~\ref{rainbowmatching}.
However, we believe that new ideas are needed to prove the case when $2k < |G| < 3 k$.

\section{Existence of rainbow matching covers} \label{sec:decom}

\begin{proof}[Proof of Theorem~\ref{thm:rainmatchigndecom}]
By colouring every missing edge in $G$ with a new colour, we may assume that $G$ is an edge-coloured complete graph on $n$ vertices with $ \Dmon (G) = t$ and colours $\{1, 2, \dots, p \}$.
For $i \le p$, let $G^i$ be the subgraph of $G$ induced by the edges of colour~$i$.
Without loss of generality, we may assume that $e(G^1) \ge e(G^2) \ge \dots \ge e(G^p)$.

For $1 \le i \le p$, suppose that we have already found a set $\mathcal{M} = \{M_1, \dots, M_{ \lfloor tn/2 \rfloor}\}$ of edge-disjoint (possiblely empty) rainbow matchings such that $\bigcup_{1 \le j \le \lfloor tn/2 \rfloor} M_j = \bigcup_{j' < i} E(G^{j'})$.
We now assign edges of $G^i$ to these matchings so that the resulting rainbow matchings $M'_1, \dots, M'_{\lfloor tn/2 \rfloor}$ contain all edges of $G^{1} \cup \dots \cup G^i$.
Define an auxiliary bipartite graph~$H$ as follows.
The vertex classes of $H$ are $E(G^i)$ and $\mathcal{M}$.
An edge $f \in E(G^i)$ is joined to a rainbow matching $M_j \in \mathcal{M}$ if and only if $f$ is vertex-disjoint from $M_j$.
If $H$ contains a matching of size $e(G^i)$, then we assign $f \in E(G^i)$ to $M_j \in \mathcal{M}$ according to the matching in~$H$.
Thus we have obtained the desired rainbow matchings $M'_1, \dots, M'_{\lfloor tn/2 \rfloor}$.
Therefore, to prove the theorem, it is sufficient to show that $H$ satisfies Hall's conditions.

Let $f \in E(G^i)$. 
%Note that if $M_j \in \mathcal{M} \setminus N_H(f)$, then $M$ is incident to~$f$.
Since $f$ is incident to $2(n-2)$ edges in $G$, $f$ is incident to at most $2(n-2)$ matchings $ M_j \in \mathcal{M}$.
Thus, 
\begin{align}
|N_H(f)| \ge |\mathcal{M}| - 2(n-2)\ge (t-4)n/2. \label{eqn:f}
\end{align}
We divide the proof into two cases depending on the value of~$i$.

\medskip
\noindent
\textbf{Case 1: $i \le \frac{(t-4)n}{4 ( t+1 ) }$.}
Let $S \subseteq E(G^i)$ with $S \ne \emptyset$.
Note that each $M_j \in \mathcal{M}$ has size at most $i-1$. 
If $S$ contains a matching of size $2i-1$, then for every $M_j \in \mathcal{M}$, there exists an edge $f \in S$ vertex-disjoint from $M_j$.
Thus, $N_H(S) = \mathcal{M}$ and so $|N_H(S)| = \lfloor tn/2 \rfloor \ge e(G^i) \ge |S|$.

Therefore, we may assume that $S$ does not contain a matching of size $2i -1$.
By Vizing's theorem, $|S| \le 2(i-1)(\Delta(G^i)+1) \le 2(i-1)(t+1)$.
By~\eqref{eqn:f} and the assumption on~$i$, we have
\begin{align*}
	|N_H(S)|
& \ge (t-4)n/2 \ge 2(i-1)(t+1) \ge |S|.
\end{align*}
Therefore, Hall's condition holds for this case.

\noindent
\textbf{Case 2: $i >\frac{(t-4)n}{4 ( t+1 ) }$.}
Since $e(G^1) \ge e(G^2) \ge \dots \ge e(G^p)$, we have $e(G^i) \le \binom{n}2/i < 2 (t+1) n / (t-4)$.
Let $S \subseteq E(G^i)$ with $S \ne \emptyset$.
By~\eqref{eqn:f} and the fact that $t \ge 11$, we have 
\begin{align*}
	|N_H(S)|
& \ge (t-4)n/2 
 \ge 2 (t+1) n / (t-4) > e(G^i) \ge |S|.
\end{align*}
Therefore, Hall's condition also holds for this case.
This completes the proof of the theorem.
\end{proof}

\section*{Acknowledgements}

The author would like to thank the referees for their helpful suggestions.

\end{document}